\documentclass[a4paper,12pt,reqno]{amsart}
\usepackage[utf8]{inputenc}
\usepackage{amssymb,amsthm}
\usepackage{enumerate}
\usepackage{graphicx}
\usepackage{color}

\newtheorem{theorem}{Theorem}[section]
\newtheorem{prop}[theorem]{Proposition}
\newtheorem{cor}[theorem]{Corollary}
\newtheorem{lemma}[theorem]{Lemma}
\newtheorem{remark}[theorem]{Remark}

\newcommand{\R}{\mathbb{R}}

\title[Number of limit cycles]
{Number of limit cycles for planar systems with invariant algebraic curves}

\author[A. Gasull]{Armengol Gasull}
\address{Departament de Matem\`{a}tiques, Edifici Cc, Universitat Aut\`{o}noma de Barcelona, 08193 Cerdanyola del Vall\`{e}s (Barcelona), Spain.}
\address{Centre de Recerca Matem\`{a}tica, Edifici Cc, Campus de Bellaterra, 08193 Cerdanyola del Vall\`{e}s (Barcelona), Spain.}
\email{gasull@mat.uab.cat}

\author[H. Giacomini]{Hector Giacomini}
\address{ Institut Denis Poisson. Universit\'{e} de Tours,  C.N.R.S. UMR 7013. 37200 Tours, France.} \email{Hector.Giacomini@lmpt.univ-tours.fr}

\date{}

\begin{document}

\begin{abstract} 
For planar polynomials systems  the existence of an invariant algebraic curve limits the number of limit cycles not contained in this curve. We present a general approach to prove non existence of periodic orbits not contained in this given algebraic curve. When the method is applied to parametric families of polynomial systems that have limit cycles for some values of the parameters, our result leads to effective algebraic conditions on the parameters that force non existence of the periodic orbits. As applications we consider several families of quadratic systems: the ones having some quadratic invariant algebraic curve, the known ones having an algebraic  limit cycle,   a  family having a cubic invariant algebraic curve and other ones.   For any quadratic system with two invariant algebraic curves we prove a finiteness result for its number of limit cycles that only depends on the degrees of these curves. We also consider some families of cubic systems having either a quadratic or a cubic invariant algebraic curve and a family of Liénard systems.
We also give a new and simple proof of the known fact that quadratic systems with an invariant parabola have at most one limit cycle. In fact, what we show is that this result is  a consequence of the similar result for quadratic systems  with an invariant straight line.
\end{abstract}

\maketitle

\section{Introduction}

The study of the number of limit cycles of  polynomial differential systems 
has attracted the interest of many researchers due to the unexpected
difficulties that it presents. In fact, even in the case of quadratic systems (QS) the problem remains open.
Perhaps the most influential works for this very particular case
have been the Coppel survey (\cite{Cop1966})   and the Ye Yanqian
{\it et alt.} book (\cite{Ye1986}).

To advance in particular subfamilies of QS people has started studying QS with additional
properties. One of these properties has been the existence of a particular invariant algebraic curve for the QS. One of the reasons is the general belief that for any polynomial system the existence of an invariant algebraic curve for it limits the number of limit cycles not contained on this curve.

 So a first celebrated family has been {\it QS with an invariant straight line}.
 With some effort people has been able to prove the existence of at most one limit
 cycle for this particular case, and also its hyperbolicity, see \cite{CZ1,CZ2,CL,Cop,Ryc}.

Afterwards the attention has moved to QS with irreducible invariant
curves of degree two. In this case it was not difficult prove that
no limit cycles coexist with hyperbolas or ellipses, although, in
the later case, the ellipse itself can be a limit cycle,  see
\cite{Che,Ye1986} and also the proofs of Theorem~\ref{th:teo2} and Proposition~\ref{th:hyper}. In case
of {\it QS having an invariant parabola} it took more effort to prove
the existence and uniqueness of the limit cycle, see \cite{Ch,ZK}.

It is remarkable that it is possible to prove that result as a direct consequence of the
same result for QS with an invariant straight line. This  gives a simple
and alternative proof to that of \cite{ZK}.

\begin{theorem}\label{th:para}
A quadratic system with an invariant parabola and a limit cycle can
be transformed into  another quadratic system with an invariant
straight line. As a consequence, quadratic systems with an invariant
parabola
 have at most one limit cycle and when it exists it is hyperbolic.
\end{theorem}

Continuing in the world of QS with invariant algebraic curves
several directions  have been appeared: studying the degrees of the
possible invariant curves, the existence of algebraic limit cycles
(limit cycles included in an invariant algebraic curve), the number
of limit cycles of families with higher degree invariant curves,
\ldots. Our second result considers the question of the non
existence of other limit cycles when  the QS has an algebraic limit
cycle. In \cite{LV1,LV2} the authors study a related problem. They give conditions on a  QS to ensure that it has at most one algebraic limit cycle.

 We face the following natural problem, see for instance
\cite{Lli,LS}: Can an algebraic and a non-algebraic limit cycle
coexist in a QS?  Although it is yet an open question, it is known
that all known QS exhibiting an algebraic limit cycle have no other
limit cycle.  A proof for the known cases at that time was given in \cite{CGL}.

 Our second main result proves with an unified method this fact
   for all the eight known families with one algebraic limit cycle. The eight families that we will study, together with the expression of their
   corresponding algebraic limit cycles  are listed in the proof of the following theorem.
   Here, we only make some few comments about them, see the nice paper \cite{AFL} for more details.
   The researchers who have found these families are: Qin; Yablonskii; Filiptsov; Chavarriga; Christopher,
   Llibre and \'Swirszcz (two cases), Chavarriga, Llibre and Sorolla; and Alberich-Carrami\~{n}ana, Ferragut and Llibre.
  
  In fact, for completeness, our proof studies all the eight cases, although from the results of \cite{AFL} it is clear these eight cases can be reduced to four of them because the other ones can be obtained from these four via suitable
 Cremona transformations (birational automorphisms) and changes of time. See more details in \cite[Fig. 10]{AFL}.

 \begin{theorem}\label{th:teo2}
    The eight families of quadratic systems detailed in the proof, and that  nowadays include all the cases of known quadratic systems with an
    algebraic limit cycle, have at most one limit cycle and when it exists it is this algebraic limit cycle.
 \end{theorem}

We do not study here the hyperbolicity of the limit cycles
considered in the above theorem. This question  is studied  in
\cite{GG}.

  The  general approach that we use to prove the eight cases of the above theorem is detailed in Section \ref{se:pre}, see Theorem \ref{th:met}, and essentially consists on a  method for effectively finding  a
kind of Lyapunov functions for  systems with an invariant algebraic curve. As we will see, these functions are actually not Lyapunov
functions because only the condition on not changing sign of its temporal derivative is required. The ideas of this approach go back to  Poincaré and have also been used in some cases in \cite{CG,CGL}. In fact, in this work  we will show that the method provides a systematic approach to study many families of planar differential systems, not necessarily QS, with invariant algebraic curves.

 Our third main result deals with QS with two generalized invariant curves and gives an upper bound for their number of limit cycles. The same question but with a different point of view was treated in \cite{CG2,CGS} but our results are different of the ones given in these works.  See for instance \cite{GGine} for several examples of polinomial systems with generalized invariant curves. The definition of these type of curves and what we mean when we say that two of them are different is deferred to the following section. Here we only comment that  they include usual invariant algebraic curves. 
	
\begin{theorem}\label{th:teo3}
	Consider a quadratic system with two different generalized invariant  curves. Let $M$ be the number of limit cycles contained in them. Then its maximum number of limit cycles is $2M+2.$
\end{theorem}	
	
As we will see, our proof starts with the result of \cite{P} that shows that if a QS has two different nests of limit cycles then one of the critical points is surrounded by a single limit cycle. A similar result, but with a bigger upper bound for the number of limit cycles, could be obtained  under our hypotheses	without using that result.

 Another key point in our proof  is to prove that all the limit cycles not contained in the generalized invariant curves are hyperbolic and have the same stability. As we will see, to prove this fact we find some algebraic relations between the characteristic exponents of the periodic orbits and some integrals involving  the cofactors of the curves. It is also possible to prove this result by using the generalized Bendixson-Dulac Theorem, see for instance \cite{GG3,GG2}.

We remark that although most probably this upper bound is not sharp (in fact, people believe that this type of QS does  not have limit cycles) it is never easy to found upper bounds for the number of limit cycles of a given QS.  We also remark that although after the work of Bamon (\cite{Bam}) the finiteness of the number of limit cycles for a given QS is known, the existence of a uniform bound for all QS is yet an open problem.

We also prove the following corollaries:

\begin{cor}\label{cor:2men}
	If we add to the hypotheses of Theorem \ref{th:teo3}  that the two generalized invariant curves do not have isolated real points, then the upper bound of the theorem can be reduce to be $2M$ and moreover all these limit cycles are nested.
\end{cor}

Notice that this corollary extends the classical result that Lotka-Volterra QS, that is QS with two non parallel invariant straight lines, do not have limit cycles. This is so because in this case $M=0$ and of course these curves do not contain real isolated points. See also Subsection  \ref{se:lv} for a different and direct proof.

The following result gives in the case of invariant algebraic curves  an upper bound only in terms of the degrees of these curves and extends again the result for Lotka-Volterra QS. 

\begin{cor}\label{cor}
Consider a quadratic system with two different invariant  irreducible algebraic curves of degrees $m_1$ and $m_2$. Then its maximum number of limit cycles  is $2([m_1/2]+[m_2/2])+2,$ where $[\,\,]$ is the integer part function. Moreover, if the two algebraic invariant curves do not have real isolated points, then the upper bound	 reduces to $2([m_1/2]+[m_2/2]).$
\end{cor}	

In Corollary \ref{co:2cl} we also prove that under some more additional conditions a QS with two generalized invariant curves have  in some cases at most two limit cycles and in other ones do not have limit cycles. We do not include the precise statements in this introduction because we would need to introduce some more notations and  definitions.

In Section \ref{se:final}, we apply Theorem \ref{th:met} to other QS to prove non existence of limit cycles,  and also  to give  conditions for non existence of limit cycles for several parametric families of QS having some invariant algebraic curve,  such that for some values of their parameters they have a limit cycle not contained in this curve.

Finally, to illustrate that Theorem~\ref{th:met} is also useful to families not necessarily quadratic we dedicate Section~\ref{se:of} to other parametric families of planar systems that have some invariant algebraic curve and limit cycles, not contained in these curves, for some values  of the parameters. Specifically, in Subsection \ref{ss:cs} we study two families of cubic systems and in Subsection \ref{ss:ls} we give some results on Liénard systems.

\section{Preliminary results}
\label{se:pre}

Consider a polynomial differential system
\begin{equation}\label{eq:gen}
    \dot x= P(x,y),\quad \dot y=Q(x,y)
\end{equation}
and denote by $\mathcal{X}=(P,Q)$ its associated vector field. Recall that it
is  said that it has an invariant algebraic curve $f(x,y)=0$ if $f$
is irreducible and  it holds that
\begin{equation}\label{eq:iac}
\dot f(x,y)=\frac{\partial f(x,y)}{\partial x}
P(x,y)+\frac{\partial f(x,y)}{\partial y} Q(x,y)=k(x,y)f(x,y),
\end{equation}
for some polynomial $k(x,y),$ called the {\it cofactor} of $f.$
%Sometimes  the above equality is also written as $X(f)=kf.$ 
Note that if the set $\{(x,y)\in\R^2\,:\, f(x,y)=0\}$ is not empty, then
it is invariant by the flow of \eqref{eq:gen}. Observe also that if
the degree of $\mathcal{X}$ (i. e., the maximum of the degrees of $P$ and $Q$)
is $n,$ then the degree of $k$ is at most $n-1.$

A simple, but key result for proving our results  will be the next theorem. As we have already explained it is used in some particular cases in \cite{CG,CGL} and here it is presented as a systematic method for studying the non existence of periodic orbits of polynomial differential systems with invariant algebraic curves.

\begin{theorem}\label{th:met}
    Consider the polynomial differential system \eqref{eq:gen} and assume
     that it has an invariant algebraic curve $f(x,y)=0$ with cofactor $k(x,y).$ For each $\alpha\in\R$ and each polynomial $g(x,y)$ define the new polynomial
    \[
    N_{\alpha,g}(x,y)=\alpha k(x,y)g(x,y)+ \frac{\partial g(x,y)}{\partial x}
    P(x,y)+\frac{\partial g(x,y)}{\partial y} Q(x,y).
    \]
    Then, if for some $\alpha$ and $g,$  $N_{\alpha,g}$ does not change sign
     and vanishes only on some algebraic curve that it is not invariant by  the
      flow, then the only limit cycles of the system (if any) are included in the invariant algebraic curve $f(x,y)=0.$

    Moreover,  if for some $\alpha$ and $g,$  $N_{\alpha,g}(x,y)\equiv0$ then $g(x,y)|f(x,y)|^\alpha,$ if it is not piecewise constant, is a
    first integral of \eqref{eq:gen} and the same conclusion holds.
\end{theorem}

\begin{proof}
    We will use the following well know fact: If for some open set $\mathcal{U}\subset\R^2,$
    there exists a class $\mathcal{C}^1$ function such that $v:\mathcal{U}\to\R,$ and
    \begin{equation}\label{eq:vdot}\dot v(x,y)=\frac{\partial v(x,y)}{\partial x}
    P(x,y)+\frac{\partial v(x,y)}{\partial y} Q(x,y)\end{equation}
does not vanish then the system \eqref{eq:gen} does not have periodic
orbits totally contained in $\mathcal{U}.$ This is so because while
the solution is in $\mathcal{U}$ the function $t\rightarrow
v(x(t),y(t)),$    where $(x(t),y(t))$ is any solution of the
differential equation, is monotonous. Clearly, this fact prevents
the existence of periodic orbits. If the  right hand side of
\eqref{eq:vdot} vanishes on some curve, but does not change it sign,
then the same holds, unless this curve is invariant by the flow.

 Since $f(x,y)=0$ is an invariant algebraic curve, each  of the connected components of $\mathcal{U}\setminus\{f(x,y)=0\}$ is invariant.
For proving the theorem we  apply the above result by taking $v(x,y)=g(x,y)|f(x,y)|^\alpha$  and   $\mathcal{U}$
any of these components. Some computations give that
\[
\dot v(x,y)= |f(x,y)|^\alpha N_{\alpha,g}(x,y)
\]
and hence the result follows under the hypotheses on $N_{\alpha,g}.$ Clearly when $N_{\alpha,g}(x,y)\equiv0$  then $\dot v(x,y)\equiv 0$ and $v$ is a first integral of \eqref{eq:gen}. Then although periodic orbits can exist, the only possible limit cycle are contained in $\{f(x,y)=0\},$ as we wanted to prove.
\end{proof}

For practical applications one idea  to prove that $N_{\alpha,g}$ does not change sign, and that works in many situations, is to introduce some parameters in the unknown function $g,$ in such a way that a clever choice of them allows to prove the existence of some $r\in\R$ and some polynomial $w$ such that   $N_{\alpha,g}(x,y)= r w^2(x,y).$   Another useful approach is to try to force that $N$ only depends on one of the variables and then impose that all its roots be double.

Now  we give two remarks with two extensions of Theorem~\ref{th:met}.	

\begin{remark}\label{re:1}
	When  system \eqref{eq:gen} has two  different invariant irreducible algebraic curves $f_i(x,y)=0, i=1,2,$ with respective cofactors $k_i(x,y), $ a similar approach can be used. For each $\alpha_1,\alpha_2\in\R$ and each polynomial $g(x,y)$ we can define the new polynomial
	\begin{align*}
	N_{\alpha_1,\alpha_2,g}(x,y)& =\big(\alpha_1 k_1(x,y)+\alpha_2 k_2(x,y)\big)g(x,y)\\&\quad+ \frac{\partial g(x,y)}{\partial x}
	P(x,y)+\frac{\partial g(x,y)}{\partial y} Q(x,y).
	\end{align*}
Then under the same hypotheses that in Theorem \ref{th:met} for this new function $N$ the same conclusions hold. 
\end{remark}

Given any smooth  function $f(x,y)$, not necessarily polynomial, we will say that $f(x,y)=0$ is a generalized invariant  curve for~$\mathcal{X}$ if $\dot f (x,y)= k(x,y)f(x,y),$ as in \eqref{eq:iac}, with $k$ also a polynomial of degree at most $n-1,$ that is, one less that the degree of the vector field. This $k$ is called its cofactor.  Given two generalized invariant curves we will say that they are different if it  do not exist any $\alpha,\beta\in\R$ such that  $|f_1(x,y)|=\alpha|f_2(x,y)|^\beta.$

\begin{remark}\label{re:2}
  Theorem \ref{th:met} and Remark \ref{re:1} also hold if in their respective statements the invariant algebraic curves are replaced by generalized invariant  curves.
\end{remark}

\section{Proof of Theorem \ref{th:para}} It is not restrictive to suppose that the
invariant parabola $\mathcal{P}$ is $f(x,y)=y-x^2=0.$ It is well known that QS  having the invariant parabola $\mathcal{P}$ can be written as
    \begin{equation}\label{eq:para}
        \begin{cases}
            \dot x=a+bx+hy +c(y-x^2)+exy=P(x,y),\\
            \dot y=2x(a+bx+hy)+d(y-x^2)+2ey^2=Q(x,y),
        \end{cases}
    \end{equation}
    where $a,b,c,d,e,h$ are arbitrary real parameters, see \cite{Ch}.
    In fact the cofactor of $f$ is $k(x,y)=-2cx+2ey+d.$ Moreover, the critical points
     on $\mathcal{P}$ are of the form $(x_0,x_0^2)$ where $x_0$ is one of the real roots of
    \[P(x,x^2)=ex^3+hx^2+bx+a=0.\] Next, we split our proof in two cases:
    \begin{enumerate}[(i)]
        \item  The above equation has no real roots. Then in particular $e=0.$
        \item The above equation has some real root, say $x=x_0.$
    \end{enumerate}

    We split again case (i)  in two subcases $c=0$ and $c\ne0$ and in both of them we will
    apply Theorem \ref{th:met} for proving the non existence of limit cycles.

    Consider first the case (i) with $c=0.$ Then, since $e=0,$ by taking $g=1$
    and $\alpha=1$ in that theorem we get that $N(x,y)=k(x,y)=d,$ and the result follows.

    For the situation (i) with $c\ne0,$ recall that $e=0.$ We consider
    \[
    g(x,y)=2ac^2+bcd+d^2h-2ch(dx-cy)\quad \mbox{and}\quad \alpha=h/c
    \]
    and recall that $k(x,y)=-2cx+d.$ Then some computations give that
    \[
    N(x,y)=\frac{h(2cx-d)^2(bc+dh)}c
    \]
    and again the result follows.

    In case (ii) we will prove that if there exists a limit cycle it is unique and hyperbolic by transforming
    these QS to new QS with an invariant straight line. For these QS, as we have already commented, the uniqueness and hyperbolicity
    follows by the results of \cite{CZ1,CZ2,CL,Cop,Ryc}.

    Following \cite{Ch} we start proving that in case (ii) it is not restrictive to
    consider $a=0$ in \eqref{eq:para}. Take the new coordinates $X=x-x_0$ and $Y=x_0^2-2x_0x+y.$ In these
    coordinates $y-x^2=0$ is transformed into $Y-X^2=0.$ Moreover, system \eqref{eq:para} writes as
    \begin{equation*}
        \begin{cases}
            \dot X= BX+HY +C(Y-X^2)+eXY,\\
            \dot Y= 2X(BX+HY)+D(Y-X^2)+2eY^2,
        \end{cases}
    \end{equation*}
    where
    \begin{align*}
        B &= b+2hx_0+3ex_0^2,\quad C = c-2ex_0,\\D &=  d-2cx_0 +2ex_0^2, \quad H = h+3ex_0.
    \end{align*}
    Hence, from now on, we consider system \eqref{eq:para} with $a=0.$
    Let $\gamma$ be a limit cycle of this system. It is totally contained in one of the six
    connected components of $\R^2\setminus\left(\mathcal{P}\cup\{x=0\}\cup\{y=0\} \right).$ This is so because $\mathcal{P}$ is invariant and moreover
    \[
    \dot x|_{x=0}=(c+h)y,\quad\mbox{and}\quad  \dot y|_{y=0}=(2b-d)x^2.
    \]
    Call $\mathcal{V}$ this connected component which, of course, must contain a critical point of index $+1.$
    Hence in particular the new birrational change of variables
    \[
    X=\frac{x}{y-x^2},\quad Y=\frac{y}{y-x^2}
    \]
    is well defined on the region $\mathcal{V}$ where $\gamma$ lies, and its inverse
    is \[x=\dfrac{Y-1}X,\quad y=\dfrac{Y(Y-1)}{X^2}.\] It transforms the parabola $y-x^2=0$ into the straight line $Y-1=0.$
    By introducing a new time $s$ such that $ds/dt=X,$ and writing $Z'=dZ/ds=XdZ/dt=X\dot Z,$ after some computations we get that
    \begin{equation*}
        \begin{cases}
            X'=X\dot X= -cX+eY+(b-d)X^2+(2c+h)XY-eY^2,\\[0.2cm]
            Y'=X\dot Y=(Y-1)\big((2b-d)X+2(c+h)Y\big),
        \end{cases}
    \end{equation*}
    which is a quadratic system with the invariant straight line $Y-1=0,$ as we wanted to prove.
\section{Proof of Theorem \ref{th:teo2}}

For each of the cases the result will be a consequence of  Theorem~\ref{th:met} by
choosing suitable $\alpha$ and $g$ in such a way that $N_{\alpha,g}(x,y)= r w^2(x,y),$ for some $r\in\R$ and some polynomial $w.$

We skip all the computations and we simply write $P,Q,$ the algebraic curve $f=0$ and
its cofactor $k,$ a suitable  constant $\alpha$ and  function $g,$ and finally the corresponding  $N=N_{\alpha,g}.$

For some of the systems some given values of the parameters have to
be omitted. For instance, the value $a=-10/3$ in Case 4.
These values could be studied separately, but for the
sake of shortness and because for them the invariant algebraic curve
does not contain any limit cycle we do not study these particular values.

\smallskip

\noindent {\bf Case 1:} Qin system with algebraic limit cycle of degree 2, see \cite{Q}:
\begin{align*} P&= -y(ax+by+c)-(x^2+y^2-1), \quad Q=x(ax+by+c),\\ f&= x^2+y^2-1,\quad k=-2x,\quad \alpha=b/2,\quad  g=c+by,\quad
    N= abx^2.
\end{align*}
We remark that the corresponding QS has only $x^2+y^2-1=0$ as a
limit cycle when   $a\ne0,$ $c^2+4(b+1)>0$ and $a^2+b^2<c^2$ but the
proof that there are no other limit cycles works for all values of
the parameters. A similar fact also holds for all the other cases.
In fact the range of parameters for which each invariant algebraic
curve contains a limit cycle is detailed in \cite{AFL}.

\bigskip

\noindent {\bf Case 2:} Yablonskii system with algebraic limit cycle  of degree 4, see \cite{Y}:
\begin{align*} P&= -4abcx-(a+b)y+3(a+b)cx^2+4xy, \\
    \begin{split}
            Q &= (a+b)abx-4abcy+(4abc^2-3(a+b)^2/2+4ab)x^2\\&\quad+8(a+b)cxy+8y^2,\end{split}
        \\f &= 
    (y+cx^2)^2+x^2(x-a)(x-b),\,\,\,
k=4(-2abc+3c(a+b)x+4y),\\ g &= 2c(a-3b)(3a-b)(y+cx^2)-ab(a+b-4x)^2,\\
\alpha &=-1/2,\quad N=-c\big(2ab(a+b)+(2ab-3a^2-3b^2)x\big)^2.&
    \end{align*}

\bigskip

\noindent {\bf Case 3:} Filipstov system with algebraic limit cycle of degree 4, see \cite{F}:
\begin{align*}
    P &=6(1+a)x+2y-6(2+a)x^2+12xy, \\
Q &=15(1+a)y+3a(1+a)x^2-2(9+5a)xy+16y^2,\\
f &=
3(1+a)(ax^2+y)^2+2y^2(2y-3(1+a)x),\\ k &=6(5(1+a)-(8+4a)x+8y),\\ \alpha &= -1/2,\quad g=2y+(3+5a)x^2,\quad N=-\big(3(1+a)x-4y\big)^2.
\end{align*}

\bigskip

\noindent {\bf Case 4:} Chavarriga system with algebraic limit cycle of degree 4, see \cite{CLS}: \begin{align*}
    P&=5x+6x^2+4(1+a)xy+ay^2,\quad
    Q=x+2y+4xy+(2+3a)y^2,\\
f&=x^2+x^3+x^2y+2axy^2+2axy^3+a^2y^4, \quad a\ne -10/3,\\
k&=2\big(5+9x+(5+6a)y\big),\quad \alpha=-\frac{7+3a}{3(10+3a)},\\
g&=-5+(21+9a)x-(35+15a)y,\\ N&=\frac{2(7+3a)}{3(10+3a)}\big(5+9x+(5+6a)y\big)^2.
\end{align*}

\bigskip

\noindent {\bf Case 5:} Chavarriga, Llibre and Sorolla system with algebraic limit cycle of degree 4, see \cite{CLS}:
\begin{align*}
    P&=2(1+2x-2ax^2+6xy),\quad
    Q=8-3a-14ax-2axy-8y^2,\\
f&=1/4+x-x^2+ax^3+xy+x^2y^2,\quad k=4(2-3ax+2y),\\
 \alpha&=1/3,\quad g=-5+3ax/2+y,\quad N=-4(2-3ax+2y)^2/3.
    \end{align*}

\bigskip

\noindent {\bf Case 6:} Christopher, Llibre and \'Swirszcz system with algebraic limit cycle of degree 5, see \cite{CJS}:
\begin{align*}
    P&=28x+2 (16-a^2) (a+12) x^2+6 (3 a-4) xy-\frac{12 }{a+4}y^2, \quad a\ne-4,\\
    Q&=2 (16-a^2)x+8 y+(16-a^2) (a+12)  xy+2 (5 a-12) y^2,\\
\begin{split}
f&=\,x^2+(16-a^2)  x^3+(a-2) x^2y-\frac{2}{a+4}x y^2-\frac{1}{4}   (4-a) (a+12)x^2 y^2\\
&\quad +\frac{8-a}{a+4}x y^3+\frac{1}{(a+4)^2}y^4+ \frac{a+12 }{a+4}xy^4-\frac{6 }{(a+4)^2}y^5,
\end{split} \\
k&=56 + 6 (16- a^2) (a+12) x + 4(13 a-24 ) y,\,
\alpha=-\frac{3+4a}{15(3+a)}, a\ne-3,\\ g&=28+(-144-192a+9a^2+12a^3)x+(42+56a)y,\\
N&=-\frac{2(3+4a)}{15(3+a)}\big(28+(576+48a-36a^2-3a^3)x+(-48+26a)y\big)^2.
\end{align*}

\bigskip

\noindent {\bf Case 7:} Christopher, Llibre and \'Swirszcz system with algebraic limit cycle of degree 6, see \cite{CJS}:
\begin{align*}
    P&=28a ( a - 30 )  x + y + 1 6 8 a^2 x^2 + 3 x y,\\
    \begin{split}
 Q&=16a(a-30)\big(14a(a-30)x+5y+84a^2x^2\big)+24a (17a-6)xy+6y^2,
\end{split}\\
    \begin{split}
f&= 48a^3(a - 30)^4x^2+ 24a^2(a - 30)^3xy+ 3a(a - 30)^2  y^2\\
& \quad+ 64a^3(a - 30)^3(9a - 4)x^3+ 24a^2(a - 30)^2(9a - 4)x^2y \\
& \quad+ 18a(a - 30)( a-2) xy^2 -7y^3 + 576a^3(a - 30)^2( a-2)^2x^4 \\
&\quad+ 144a^2(a - 30)(a - 2)^2x^3y + 27a(a - 2)^2 x^2y^2\\
&\quad- 3456a^3(a - 30)(a-2)^2( 2a+3)x^5- 432a^2(a - 2)^2( 2a+3)x^4y  \\
&\quad+ 3456a^3(a - 2)^2( a+12)( 2a+3)x^6,
\end{split}\\
k&=168a (a-30)   + 1008 a^2 x + 18 y,\quad \alpha=-1/3,\\
\begin{split}
g&=-16a(a-30)^2-24a(a-30)(7a-30)x\\&\quad-72a(360-78a+5a^2)x^2+3(30+a)y,\end{split}\\
N&= 896a^2(a-30)(a-30+6ax)^2.
\end{align*}
We remark that there is a misprint when this system in written in~\cite{AFL}. There the coefficient of $y$ in $Q$ is typed as $516a(a-30)$  instead of the correct value  $80a(a-30).$

\bigskip

\noindent {\bf Case 8:} Alberich-Carrami\~{n}ana, Ferragut and Llibre
system with algebraic limit cycle of degree 5, see \cite{AFL}:
\begin{align*}
    P&=-8x+\frac a 2(a-16) y- (5a-64) x^2   +\frac a 8 (a^2-256) xy,\\
    Q&= - 28 y+\frac{24}{a} x^2-3 (3 a-32) x y +\frac a 4 (a^2-256) y^2,\\
\begin{split}
f&=\, a  y^2-4x^2y+\frac a 2  (a-12) x y^2 -\frac{a^2}4 (a-16)y^3   +  \frac{4}{a}x^4\\
&\quad+(24-a)x^3 y+ \frac{a}{16} (a^2-256)x^2y^2 -\frac{24}{a} x^5 +(a+16) x^4y,
\end{split}\\
k&=-56 -2(13a-152) x +\frac{3a}4 (a^2-256) y,\quad \alpha=\frac{26-4a}{15(a-2)},\, a\ne2,\\
g&= 112+56(2a-13)x+3a(2a-13)(a-16)y,\\
N&=\frac{2a-13}{60(a-2)}\big(-224+(1216-104a)x+a(3a^2-768)y\big)^2.
\end{align*}

\section{Proof of Theorem \ref{th:teo3}}	

This section is devoted to prove Theorem~\ref{th:teo3}, Corollary \ref{cor} and other related results. When we say that the sign of a  number, say $h,$ gives the stability of a limit cycle or a critical point we mean that when the sign of $h$ is positive (resp. negative) this object is a repellor (resp. an attractor).

\begin{proof}[Proof of Theorem \ref{th:teo3}]
	From the classical results on QS (\cite{Cop1966,Ye1986}) and the results of \cite{P} we know that each limit cycle of a QS surrounds exactly one critical point, that must be of focus type. Moreover, there are at most two nests of concentric limit cycles and in this case one of the nests is formed by a single limit cycle. In other words,  the configurations of limit cycles for a QS are either $L$ limit cycles surrounding a  focus, or  $L$ limit cycles surrounding a focus and one more limit cycle surrounding another focus. Moreover, from \cite{Bam} and for a given QS, $L$ is finite.

	Let $f_i(x,y)=0, i=1,2$ be the two different generalized invariant  curves, and let $k_i(x,y)=a_i+b_ix+c_i y, i=1,2$ be their respective cofactors. Let $\operatorname{div}(\mathcal{X}(x,y))=a+bx+cy$  be the divergence of the vector field associated to the QS.
	It is well known (\cite{Ye1986}) that the characteristic exponent of a limit cycle $\gamma$ is 
	\[
	h=\int_0^T \operatorname{div}(\mathcal{X})(x(t),y(t))\,{\rm d} t= \int_0^T \big(a+bx(t)+cy(t)  \big)\,{\rm d} t,
	\]
	where 	$(x(t),y(t))$ is the time parameterization of $\gamma$ and $T$ is its period. Recall that if $h\ne0$ then $\gamma$ is hyperbolic and its sign gives its stability.
	
	For convenience we introduce the sets  $\mathcal{Z}_i=\{f_i(x,y)=0\}, i=1,2,$ $\mathcal{Z}_{1,2}=\mathcal{Z}_1\cup\mathcal{Z}_2$ and  also  two real numbers $X$ and $Y,$ as
	\[
	X=\int_0^T x(t)\,{\rm d} t, \quad  Y=\int_0^T y(t)\,{\rm d} t \quad\mbox{and}\quad T=\int_0^T {\rm d} t.
	\]
	
	Let us study $h$ when  $\gamma$ is not included in  $\mathcal{Z}_{1,2}.$ In this case, by integrating the equalities $\dot f_i(x,y)/f_i(x,y)=k_i(x,y)$ over $\gamma$ it holds that
	\[
	\int_0^T k_i(x(t),y(t))\,{\rm d} t= \int_0^T \big(a_i+b_ix(t)+c_iy(t)  \big)\,{\rm d} t=0, \quad i=1,2.
	\] 
	By joining all the above equalities, for $\gamma\cap \mathcal{Z}_{1,2}=\emptyset,$ it holds that
	\[
	A\left(\begin{matrix}T\\X\\Y
	\end{matrix}\right)=\left(\begin{matrix}h\\0\\0
	\end{matrix}\right),\quad \mbox{where}\quad A=\left(\begin{matrix}a&b&c\\a_1&b_1&c_1\\a_2&b_2&c_2
	\end{matrix}\right).
	\]
	When $(b_1c_2-b_2c_1)\det(A)\ne0$ then \[h=\frac{\det(A)}{b_1c_2-b_2c_1}T\ne0\] and as a consequence all periodic orbits not contained in $\mathcal{Z}_{1,2}$ are hyperbolic limit cycles and with the same stability. In particular, only one limit can exist in the region delimited by two consecutive limit cycles contained in $\mathcal{Z}_{1,2},$ or between the critical point and the first limit cycle, or surrounding the last limit cycle of the nest. In any case, the maximum number of limit cycles that surround one focus is 1, while surrounding the other one the maximum number is $M$ contained in $\mathcal{Z}_{1,2}$ and $M+1$ not in $\mathcal{Z}_{1,2}.$ This makes a total of $2M+2$ limit cycles, as we wanted to see.
	
	To end the proof we need to study the cases whether $\det(A)(b_1c_2-b_2c_1)=0.$ In this situation we will prove that the QS does not have limit cycles outside $\mathcal{Z}_{1,2}.$   
	
	Assume first that $\operatorname{det}(A)=0.$ Then, there exist real numbers $u,v$ and $w$ such that $u k_1(x,y)+v k_2(x,y)+ w \operatorname{div}(\mathcal{X})\equiv0.$ By using the well known formula
	\begin{multline*}
		\operatorname{div}\Big(|f_1(x,y)|^p|f_2(x,y)|^q \mathcal{X}\Big) \\=|f_1(x,y)|^p|f_2(x,y)|^q\Big(p k_1(x,y)+q k_2(x,y)+ \operatorname{div}(\mathcal{X})\Big)
	\end{multline*}
	and by taking $p=u/w$ and $q=v/w,$ when $w\ne0,$ we get that the divergence is identically zero and the QS can not have limit cycles outside $\mathcal{Z}_{1,2}.$ When $w=0$ we can apply Remark \ref{re:1} with $f_1,f_2,$ $\alpha_1=u,$ $\alpha_2=v$ and $g(x,y)\equiv1$ to prove the non existence of limit cycles outside $\mathcal{Z}_{1,2}$ because $N_{\alpha_1,\alpha_2,g}(x,y)\equiv0$  and $|f_1(x,y)|^p|f_2(x,y)|^q$ is not identically constant.
	
	Finally, when $b_1c_2-b_2c_1=0,$ notice that $c_2k_1(x,y)-c_1k_2(x,y)=a_1c_2-c_1a_2.$ Hence we can apply again Remark \ref{re:1} with $f_1,f_2$ and now $\alpha_1=c_2,$ $\alpha_2=-c_1$ and $g(x,y)\equiv1,$ obtaining that $N_{\alpha_1,\alpha_2,g}(x,y)\equiv a_1c_2-c_1a_2.$ Therefore, the non existence of limit cycles outside $\mathcal{Z}_{1,2}$ follows again.	
\end{proof}

\begin{proof}[Proof of Corollary \ref{cor:2men}] In this proof we will use the same notations that in the proof of Theorem \ref{th:teo3}. Let us show first that the new hypothesis force the QS to have at most one critical point outside the generalized invariant curves. We know that $\dot f_i(x,y)=k_i(x,y)f_i(x,y),$ $i=1,2.$ Let $(\bar x,\bar y)$ be a critical point not contained in these curves. Since   $\dot f_i(x,y) |_{(x,y)=(\bar x,\bar y)}=0$  and $f_i(\bar x,\bar y)\ne0,$ we get that $k_i(\bar x,\bar y)=0,$ $i=1,2.$ Writing $k_i(x,y)=a_i+b_ix+c_i y,$ $i=1,2,$ we obtain that $(\bar x,\bar y)$ must be a solution of the lineal system
	\begin{equation}\label{eq:pc}
a_1+b_1\bar x+c_1\bar y=0,\quad 	a_2+b_2\bar x+c_2\bar y=0.
	\end{equation}
Then, either the system does not have solution,  or it has exactly one solution, or it has infinitely many. In 	the first two situations we are done. In the third one, as in the proof of Theorem \ref{th:teo3}, we can apply Theorem \ref{th:met} to prove that the QS does not have limit cycles because both cofactors  are proportional.

Since the QS does not have isolated critical points in the generalized invariant curves  we have already proved that all the limit cycles are nested and surround $(\bar x,\bar y).$ Hence, arguing as in the proof of Theorem \ref{th:teo3}, the maximum number of limit cycles is $2M+1.$ 

To decrease by one this new upper bound we will prove that between the critical point $(\bar x,\bar y)$ and the first limit cycle of the nest belonging to $\mathcal{Z}_1\cup\mathcal{Z}_2$ the QS does not have limit cycles. Assume that such a  $\gamma$ do exist. We will prove that $\gamma$ and  the critical point $(\bar x,\bar y)$ have the same  stability giving rise to a contradiction.  By the proof of Theorem~\ref{th:teo3} we know that we can restrict our attention to the case $(b_1c_2-b_2c_1)\det(A)\ne0$ and that the stability of this $\gamma$ is given by the sign of $\det(A)T/(b_1c_2-b_2c_1).$ To study the stability of the focus point $(\bar x,\bar y)$ it suffices to know the sign of $$d=\operatorname{div}(\mathcal{X})(\bar x,\bar y)=a+b\bar x+c\bar y.$$ By using this equation and \eqref{eq:pc} we get that
\[
A\left(\begin{matrix}1\\\bar x\\\bar y
\end{matrix}\right)=\left(\begin{matrix}d\\0\\0
\end{matrix}\right)
\]
and as a consequence $d=\det(A)/(b_1c_2-b_2c_1),$ proving the desired result. 
\end{proof}

\begin{proof}[Proof of Corollary \ref{cor}] By Theorem \ref{th:teo3} it suffices to prove that $M=[m_1/2]+[m_2/2].$ To prove this equality we will show that when $f(x,y)=0$ is an algebraic invariant curve of degree $m$ of a QS then its maximum number of nested ovals is $[m/2].$ This is a simple consequence of the fact that a straight line passing by a point surrounded by all these ovals cuts the curve $f$ at most at $m$ points and that each oval gives rises at least to two of these cuts.  In fact, exactly to two cuts because it is known that periodic orbits of QS are convex. The upper bound of $2([m_1/2]+[m_2/2])$ limit cycles when the invariant algebraic curves do not contain real isolated  points follows from Corollary \ref{cor:2men}.
\end{proof}	

 In \cite{CGrau} it is proved that there are QS with invariant algebraic curves of arbitrarily high degree and not being Darboux integrable. Their example is $\dot x=1,$ $\dot y=2m+2xy+y^2$ and the curve has degree $m+1$. For other similar examples see for instance \cite{CL2,MO}.  Nevertheless, under generic conditions on the QS, the higher degree of the invariant algebraic invariant curves is 4, see \cite{Car}. In fact, in that paper the author proves that if a polynomial system of degree $n$ does not have dicritic critical points (see the paper for a definition)  then the maximum degree of any invariant algebraic curve is $n+2.$  By using this result and Corollary \ref{cor} we obtain a new result.
	
	\begin{cor}\label{cor-nou}
		Consider a quadratic system with two different invariant  irreducible algebraic curves and without dicritic critical points. Then its maximum number of limit cycles  is ten. Moreover, if the two algebraic invariant curves do not have real isolated points, then this upper bound	reduces to eight.
	\end{cor}
	
	\begin{proof} From the results of  \cite{Car} we know that the maximum degrees of the two invariant algebraic curves is four. Then, taking $m_1=m_2=4$ 
		in Corollary \ref{cor} the result follows.
	\end{proof}

Similarly, if we consider the classes of QS studied in \cite{LV1,LV2}, which have at most one algebraic limit cycle, we have the following result because we can apply Theorem \ref{th:teo3} and Corollary \ref{cor:2men} with $M=1$. 

	\begin{cor}\label{cor-nou2}
	Consider a quadratic system with two different invariant  irreducible algebraic curves and satisfying the hypotheses of \cite{LV1,LV2}. Then its maximum number of limit cycles  is four. Moreover, if the two algebraic invariant curves do not have real isolated points, then this upper bound	reduces to two.
\end{cor}

Recall that from the proof of Theorem~\ref{th:teo3} we know that under its hypotheses  all the limit cycles of the QS that are not in generalized invariant curves are hyperbolic and have the same stability. In the following result we prove that something similar  holds for all limit cycles contained in $\mathcal{Z}_{1}$ and also for the ones contained in~$\mathcal{Z}_{2}.$ Moroever, when these three stabilities coincide the upper bound for the number of limit cycles drastically decreases.

\begin{cor}\label{co:2cl} Assume that the hypotheses of Theorem \ref{th:teo3} hold. Let $k_i(x,y)=a_i+b_ix+c_i y$  be the cofactors of the generalized invariant  curves of the QS, $f_i(x,y)=0,$ $i=1,2$ and let  $\operatorname{div}(\mathcal{X}(x,y))=a+bx+cy$ be its divergence. Assume that $\Delta_{1,2}\big(\Delta_{1,2}-\Delta_2\big)\big(\Delta_{1,2}+\Delta_1)\operatorname{det}(A)\ne0,$
where
\[
\Delta_{1,2}=b_1c_2-b_2c_1\quad\mbox{and}\quad \Delta_i=bc_i-cb_i,\, i=1,2.
\]	
Then it holds that
\begin{enumerate}[(i)]
	\item  The stability of the limit cycles contained in $\{f_1(x,y)=0\}$ is given by the sign of $(\Delta_{1,2}-\Delta_2)\operatorname{det}(A).$ 
	\item  The stability of the limit cycles contained in $\{f_2(x,y)=0\}$ is given by the sign of $(\Delta_{1,2}+\Delta_1)\operatorname{det}(A).$ 
	\item  The stability of the limit cycles not contained in $\{f_1(x,y)=0\}\cup\{f_2(x,y)=0\}$ is given by the sign of $\Delta_{1,2}\operatorname{det}(A).$ 
\end{enumerate}
Moreover, if the three quantities $\Delta_{1,2},$ $\Delta_{1,2}-\Delta_2$ and $\Delta_{1,2}+\Delta_1$ have the same sign, then the maximum number of limit cycles of the QS is two. Furthermore, in this last situation, if the two generalized invariant curves do not contain isolated real points then the QS does not have limit cycles.
\end{cor} 

\begin{proof} In this proof we will use the same notations that in the proof of Theorem \ref{th:teo3}. Recall that if a limit cycle does not intersect $\mathcal{Z}_{1,2}$ then its characteristic exponent is $h=\operatorname{det}(A)T/\Delta_{1,2},$  that is not zero under our hypotheses.

Let $\gamma$ be a limit cycle contained in $\mathcal{Z}_{1}.$ By using the results of \cite{GG} we know  that 
\[
h=\int_0^T \operatorname{div}(\mathcal{X})(x(t),y(t))\,{\rm d} t=\int_0^T k_1(x(t),y(t))\,{\rm d} t.
\]
Equivalently
	\[
h=\int_0^T \big(a+bx(t)+cy(t)  \big)\,{\rm d} t=\int_0^T \big(a_1+b_1x(t)+c_1y(t)  \big)\,{\rm d} t.
\]
The above equalities write as $h=aT+bX+cY=a_1T+b_1X+c_1Y.$ Moreover since $\gamma\cap\mathcal{Z}_2=\emptyset,$ $a_2T+b_2X+c_2Y=0.$  In short,  for $\gamma\subset\mathcal{Z}_{1}$ it holds that
\[
A\left(\begin{matrix}T\\X\\Y
\end{matrix}\right)=\left(\begin{matrix}h\\h\\0
\end{matrix}\right).
\] 
 Then,
\[
h= \frac{\det(A)}{\Delta_{1,2}-\Delta_2}T\ne0.
\]
Similarly, if $\gamma\subset\mathcal{Z}_2,$ 
\[
h= \frac{\det(A)}{\Delta_{1,2}+\Delta_1}T\ne0.
\]
Hence, under our hypotheses, all possible limit cycles share stability and as a consequence, at most one limit cycle can surround each of the critical points. In short, 2 is the maximum number of limit cycles, as we wanted to prove.

The non existence of limit cycles when the generalized invariant curves do not contain real isolated  points follows from Corollary \ref{cor:2men}.
\end{proof}

 We remark that the results of Theorem \ref{th:teo3} and its corollaries  could be extended to polynomial vector fields of degree $n$ having enough different generalized invariant curves (in fact $n(n+1)/2 -1$ curves). It can be proved that all the limit cycles outside of these curves have the same stability and  an upper bound of how many of them possesses the system only depends on the number of limit cycles contained in the given curves and on  their distribution. Also, generically, the stabilities of the limit cycles contained in the generalized invariant curves can be explicitly obtained. Notice also that the possible configurations of limit cycles are much more complicated than in the QS case.

\section{More results on quadratic systems}\label{se:final}

In this section we collect a miscellany of results about QS where the  common point is that their proofs are based on the method introduced in Theorem \ref{th:met}.

\subsection{Quadratic systems with an invariant hyperbola}

As we have already explained, QS with an invariant hyperbola do not have limit cycles, see \cite{Che}. We include a simple proof based on our approach. We consider that the hyperbola is $f(x,y)=x^2-y^2-1=0,$ and we take the normal form considered in \cite{G},
 \begin{equation}\label{eq:hyper}
	\begin{cases}
		\dot x=y(a+bx+cy)+u(x^2-y^2-1)=P(x,y),\\
		\dot y=x(a+bx+cy)+v(x^2-y^2-1)=Q(x,y).
	\end{cases}
\end{equation}
Then its cofactor is $2(ux-vy).$ 

\begin{prop}\label{th:hyper}
Quadratic systems with an invariant hyperbola do not  have limit cycles.
\end{prop}

\begin{proof}
	We use the same approach and notation that in the proof of Theorem \ref{th:teo2}. It is not restrictive to consider that the QS is written as in \eqref{eq:hyper}. When $u^2\ne v^2$ we take in Theorem \ref{th:met},
$$
g(x,y)=	a^2(v^2-u^2)^2+ a(v^2-u^2)(uc+vb)(vx-uy),\,\, \alpha=\frac{uc+vb}{2(v^2-u^2)}.
$$	
Then
\[
N(x,y)=a(uc+vb)(ub+vc)(ux-vy)^2
\]
and the result follows for this case.

When $u^2=v^2=0,$ trivially the QS does not have limit cycles. Finally, when $v=\sigma u,$ with $\sigma\in\{+1,-1\},$ it holds that the straight line $f_2(x,y)=y-\sigma x=0$ is also invariant. Observe that the critical points of the QS must satisfy
\[
0=yQ(x,y)-xP(x,y)=\sigma u(y-\sigma x)(x^2-y^2-1).
\]
Hence all the critical points lay on invariant algebraic curves, and since any periodic orbit $\gamma$ must surround one of these points, such a $\gamma$ can not exist by the theorem of uniqueness of solutions for these differential systems.
\end{proof}

\subsection{Lotka-Volterra QS}\label{se:lv} From the classical Bendixson-Dulac theorem, Bautin already proved that QS with two non parallel invariant straight lines do not have limit cycles, see for instance  \cite{Bau} or \cite[Sec. 1.2]{GG3}. The point in that proof is to use a Dulac function of the form $|x|^\alpha|y|^\beta.$ Here we will present a different proof by using the extension of Theorem~\ref{th:met}  given in Remark~\ref{re:1}.

\begin{prop}\label{pr:lv}
	The Lotka-Volterra QS
	 \begin{equation}\label{eq:lv}
		\begin{cases}
			\dot x=x(a_0+a_1x+a_2y),\\
			\dot y=y(b_0+b_1x+b_2y)
		\end{cases}
	\end{equation}
	does not have  limit cycles.
\end{prop}

\begin{proof}
	We define several quantities associated to \eqref{eq:lv}. Set $\Delta_{i,j}= a_ib_j-a_jb_i,$ $D=a_{{0}}a_{{1}}b_{{2}}-a_{{0}}b_{{1}}b_{{2}}-a_{{1}}a_{{2}}b_{{0}}+a_{{
			1}}b_{{0}}b_{{2}}$ and
		\[
		C=-a_{{0}}a_{{1}}b_{2}^{2}+a_{{0}}b_{{1}}b_{2}^{2}+2\,a_{{1}}a_{
			{2}}b_{{0}}b_{{2}}-a_{{1}}b_{{0}}b_{2}^{2}-a_{{2}}^{2}b_{{0}}b_{
			{1}}.
		\]
First, let us prove that when $a_2b_2\Delta_{0,2}\Delta_{1,2}=0$ the system does not have periodic orbits. When $a_2=0$ the first differential equation does not depend on $y,$ that is, $\dot x=x(a_0+a_1x),$ and hence the behaviour of $x(t)$ can not be periodic unless it is identically constant. Something similar happens when $b_2=0.$ When $\Delta_{1,2}=0$ system \eqref{eq:lv} either does not have critical points outside the axes or it has a full line of critical poins. In any case, periodic orbits are not possible. Similarly, when $\Delta_{0,2}=0$ the same happens.	

When $a_2b_2\Delta_{0,2}\Delta_{1,2}\ne0$ we will apply Remark \ref{re:1} with $f_1(x,y)=x,$ $f_2(x,y)=y,$ $k_1(x,y)=a_0+a_1x+a_2y,$ $k_2(x,y)=b_0+b_1x+b_2y,$
\begin{align*}
\alpha_1=\frac{b_2 C}{a_2\Delta_{0,2}\Delta_{1,2}},\quad  \alpha_2=\frac{b_2(a_2-b_2) \Delta_{0,1}}{\Delta_{0,2}\Delta_{1,2}}
\end{align*}	
and
\[
g(x,y)=b_2C\Delta_{1,2}x-a_2b_2(a_2-b_2)\Delta_{1,2}\Delta_{0,2}y-C(a_2-b_2)\Delta_{0,2}.
\]	
Straightforward computations give that
\[
N_{\alpha_1,\alpha_2,g}(x,y)=\frac{b_2(a_2-b_2)CD}{a_2\Delta_{0,2}\Delta_{1,2}}\big(\Delta_{1,2}x+\Delta_{0,2}\big)^2.
\]
Hence, when $b_2(a_2-b_2)CD\ne0$ the QS does not have periodic orbits. When $b_2(a_2-b_2)CD=0$ it can have periodic orbits but not limit cycles because it has a smooth first integral. For instance, the classical Lotka-Volterra system, that corresponds to $a_1=b_2=0,$ $a_0>0, b_1>0$ and $a_2<0,b_0<0,$ has a global center in the first quadrant. Notice that in this case $D=0.$
\end{proof}

\subsection{A QS with an invariant algebraic curve of degree four} In \cite{GL} the authors list QS having some classical invariant algebraic curve. Our method applies to several of them. As an example we consider the QS
\begin{equation}\label{eq:cla}
	\begin{cases}
		\dot x=- 2 k_1 r x+ 4 k_2 r y  + k_1 x^2 - 8 k_2 x y,\\
		\dot y=-3 k_1 r y+3 k_2 x^2  + 2 k_1 x y - 16 k_2 y^2,
	\end{cases}
\end{equation}
that has the invariant algebraic curve $f(x,y)= - 2 r x^3 + x^4 + 4 r^2 y^2,$ called  curve Antiversiera, with cofactor $k(x,y)=-6 k_1 r +4 k_1 x -32 k_2 y.$ 

\begin{prop}
The  QS \eqref{eq:cla} does not have limit cycles.
\end{prop}

\begin{proof}
By using Theorem \ref{th:met} with $\alpha=-1/4$ and $g(x,y)=r-2x$ we get that $N_{\alpha,g}(x,y)=\frac32k_1r^2.$  Hence, when $r k_1\ne0$ it does not have periodic orbits. When $r k_1=0$ it can have periodic orbits but not limit cycles because  it has the first integral $g^4(x,y)/f(x,y).$ For instance, when  $k_1=0$ and $r\ne0$ the QS has a center.
\end{proof}

\subsection{A QS with an invariant algebraic curve of degree twelve} In \cite{CLS} the authors prove that the QS
\begin{equation}\label{eq:12}
	\begin{cases}
		\dot x=1 +x^2 +xy,\\
		\dot y=\dfrac{57}2-\dfrac{81}2x^2+3y^2,
	\end{cases}
\end{equation}
has the invariant algebraic curve $$f(x,y)= -5488\,{y}^{4}-32\,x \left( 35125\,{x}^{2}-1029 \right) {y}^{3}+
f_2(x){y}^{2}+f_1(x) y+f_0(x),$$
where
\begin{align*}
	f_2(x)& = -375000 {x}^{6}+5058000 {x}^{4}-711288 {x}^{2}-98784,\\
	f_1(x)&= 8 x \left( 1953125 {x}^{8}+140625 {x}^{6}+2932875
	{x}^{4}-1515789 {x}^{2}+40284 \right),\\
	f_0(x)&= 48828125 {x}^{12}-23437500
	{x}^{10}+41343750 {x}^{8}-97906500 {x}^{6}\\&\quad+71546517 {x}^{4}-
	7246584{x}^{2}-442368,
\end{align*}
with cofactor $k(x,y)=12(x+y)$ and that it is not Liouvillian integrable. We show that it does not have periodic orbits.

\begin{prop}\label{th:hyper}
	The quadratic system  \eqref{eq:12} does not have periodic orbits.
\end{prop}

\begin{proof}
	A detailed study of the curve $f(x,y)=0$ shows that it does not contain closed ovals, although it contains the isolated points $(1,-2)$ and $(-1,2).$ In fact, these two points are the unique critical points of the QS, and are foci with opposite stabilities. To prove that the QS does not have periodic orbits in $\R^2\setminus\{f(x,y)=0\}$ we apply Theorem \ref{th:met} with $g(x,y)=x$ and $\alpha=-1/12.$ Then $N_{\alpha,g}(x,y)\equiv 1$ and the result follows. 
\end{proof}

\subsection{Nonexistence of limit cycles for QS with an invariant parabola}

We apply Theorem \ref{th:met} to give a simple test for non existence of limit cycles for the family of QS with an invariant parabola~\eqref{eq:para},
 \begin{equation*}
	\begin{cases}
		\dot x=a+bx+hy +c(y-x^2)+exy,\\
		\dot y=2x(a+bx+hy)+d(y-x^2)+2ey^2.
	\end{cases}
\end{equation*}
Now, we prove the following result:

\begin{prop}\label{th:para-ne}
Associated to the QS \eqref{eq:para}, define
\begin{align*}
	\Delta=\Big(& 16{a}^{2}{e}^{3}+\big(( -8bc-16hb-12cd-8dh ) a+ ( 2b-d)^{3}\big)e^2\\&-2(c+h)\big(4c ( c+2h ) a-8{b}^{2}c+6bcd-4bdh-3c{d}^{2}   \big)e\\& +8c( c+h ) ^{2}( bc+dh)\Big)\Big( (2b-d)e+2c^2+2hc\Big) e.
\end{align*}
Then if $\Delta\ge0,$ system \eqref{eq:para} does not have periodic orbits. 
\end{prop}

\begin{proof}
Recall that for the above QS the invariant parabola is $f(x,y)=y-x^2=0,$ and its cofactor  $k(x,y)=-2cx+2ey+d.$ For shortness we will write $\mathbf{A}$ to refer to the set of parameters $a,d, c, d, e, h$ and $\mathbf{G}$ to refer to the set of parameters $g_0,g_1,g_2,\alpha.$ If in Theorem \ref{th:met} we take $g(x,y)=g_0+g_1x+g_2y,$ it holds that
\[
N:=N_{\alpha,g}(x,y)= \sum_{0\leq i+j\leq2}n_{i,j}(\mathbf{A},\mathbf{G}   )x^iy^j,
\]
where $n_{i,j}$  are polynomials in their variables.
The idea is to search a condition among the parameters  $\mathbf{A}$ in such a way that the free parameters $\mathbf{G}$ can be written in terms of the parameters  $\mathbf{A},$ $\mathbf{G}=\mathbf{G}(\mathbf{A}),$  and moreover
\begin{equation}\label{eq:con}
N= \sum_{0\leq i+j\leq2}n_{i,j}(\mathbf{A},\mathbf{G}(\mathbf{A})   )x^iy^j=r\left(\sum_{0\leq i+j\leq1}w_{i,j}(\mathbf{A})x^iy^j\right)^2,
\end{equation}
where $w_{i,j}$ are polynomials on the variables $\mathbf A$ and $r\in\R.$

We claim that this condition is $\Delta\ge0.$ From this claim, the above equality, and Theorem \ref{th:met}, our result follows.

Let us give some hints for the proof of the claim. We skip the details and the cumbersome expressions.  A first observation  is that a necessary condition for the existence of $r$ and the polynomials $w_{i,j}$ in \eqref{eq:con} is that the discriminant ($\operatorname{disc}$) of $N,$ either as a quadratic polynomial in $x,$ or as a quadratic polynomial in $y,$ are identically zero. For instance, the condition
\begin{equation}\label{eq:dis}
\operatorname{disc}_y(N)=s_0(\mathbf{A},\mathbf{G} )+s_1(\mathbf{A},\mathbf{G} )x+s_2(\mathbf{A},\mathbf{G} )x^2\equiv 0
\end{equation}
has to be satisfied.
The equation $ s_1(\mathbf{A},\mathbf{G} )=0$ is of the form
\[
 s_1(\mathbf{A},\mathbf{G})=u_0(\mathbf{A},\mathbf{G})+ u_1(\mathbf{A},\mathbf{G})g_0=0,
\] 
for some polynomials $u_0$ and $u_1$ that do not depend on $g_0.$ Hence, when $u_1(\mathbf{A},\mathbf{G})\ne0,$
the value $g_0$ can be obtained from all the other parameters  $\mathbf{A}$ and $\mathbf{G}.$  By using this value of $g_0$ we get that the equality~\eqref{eq:dis} writes as
\[
\operatorname{disc}_y(N)=t_0(\mathbf{A},\mathbf{G} )+t_2(\mathbf{A},\mathbf{G} )x^2\equiv 0,
\]
where the parameter $g_0$ is no more in $\mathbf{G},$ $t_0$ is a rational function, with $u_1(\mathbf{A},\mathbf{G})$ in its denominator,  and $t_2$ a polynomial one. Next, studying the system of equations
\[
t_0(\mathbf{A},\mathbf{G} )=0, \quad t_2(\mathbf{A},\mathbf{G} )=0,
\]
we  conclude that it has solutions if
\[
( -2be-2{c}^{2}-2ch+ed) g_1+ ( 4ae-2c
d-2dh)g_2=0
\]
and $\alpha$ satisfies a quadratic polinomial equation of the form
\begin{equation}\label{eq:fff}
v_0(\mathbf{A})+v_1(\mathbf{A})\alpha+v_2(\mathbf{A})\alpha^2=0.
\end{equation}
Hence our problem, when  $u_1(\mathbf{A},\mathbf{G})\ne0,$  has been essentially reduced to the existence of a real value $\alpha$ satisfying the above quadratic equation. Finally, using that if we have a suitable set of values $g_0,g_1$ and $g_2,$ for which \eqref{eq:con} is satisfied,  the same holds for $\beta g_0,\beta g_1$ and $\beta g_2,$ for any $\beta\in\R,$ we can eliminate the presence of denominators during all the process and the condition $\operatorname{disc}_y(N)\equiv 0$ follows for all the values of the parameters.  The  discriminant of \eqref{eq:fff} with respect to $\alpha$ is precisely $4((2b-d)e+2c^2+2hc)^2\Delta,$ and so, a suitable $\alpha$ exists when $\Delta\ge0.$
\end{proof}

Notice the non existence result of Proposition~\ref{th:para-ne} also provides, through the birrational transformation used in its proof,  conditions for non existence of limit cycles of some QS with an invariant straight line.

For some cases of QS with an invariant straight line a direct proof of non existence of limit cycles for some values of the parameters also can be addressed with our approach. Next results presents an example for the simple 1-parametric subcase,
 \begin{equation}\label{eq:qsl}
	\begin{cases}
		\dot x=dx-y+\frac14x^2+\frac15xy-\frac13y^2,\\
		\dot y=x(1+y),
	\end{cases}
\end{equation}
that has the invariant straight line $f(x,y)=1+y=0$ with cofactor $k(x,y)=x.$
The results of \cite[Thm 2']{H}, only apply to QS without critical points on this invariant line and  only one singularity at infinity, the one corresponding  to this line. For our system these hypotheses correspond  to $d\in(1/5-\sqrt{6}/3, 1/5+\sqrt{6}/3) .$   By using this result and the properties of the rotated families of vector fields, if $d^\pm=(3\pm\sqrt{6})/15,$  it holds that the system has a limit cycle only when $d\in(0,d^-),$ where $d^-\approx0.0367,$ and it surrounds the origin or when $d\in(d^+,3/5),$ where  $d^+\approx0.3633,$ and it surrounds $(0,-3)$ and always it is unique and  hyperbolic. We prove a much weaker result, but with a simple proof.

\begin{prop} System  \eqref{eq:qsl}, with $d\ge 1/5,$ does not have periodic orbits surrounding the origin.
\end{prop}

\begin{proof}
	We will apply Theorem \ref{th:met} with $\alpha=0$ and $g(x,y)=g_0(y)+g_1(y)x+g_2(y)x^2,$ and a slightly different approach that in previous cases considered in this paper, for studying when $N_{0,g}$ does not change sign. Notice that with this point of view, in this case there is no need to introduce the parameter $\alpha$ in the approach because the term $|f(x,y)|^\alpha$ introduced in the proof of Theorem \ref{th:met} reduces to $(y+1)^\alpha$ in the region $y+1>0$ and this function can be thought to be already contained in the functions $g_i(y).$ If we compute $N=N_{0,g}$ we obtain
	\begin{align*}
		N(x,y)&=\Big((1+y)g_2'(y)+\frac12g_2(y)\Big)x^3\\&\quad+\Big((1+y)g_1'(y)+\frac14g_1(y)+\Big(2d+\frac25y\Big)g_2(y)\Big)x^2\\&\quad+\Big((1+y)g_0'(y)+\Big(d+\frac15y\Big)g_1(y)-2y\Big(1+\frac13 y\Big)g_2(y)\Big)x\\&\quad-\frac13y(y+3)g_1(y).
	\end{align*}
\end{proof}
To impose that $N$ does not depend on $x$ we have to solve three ordinary differential equations, one for each $g_i, i=0,1,2.$ By imposing that $g_2(0)=1$ and $g_1(0)=0$ we obtain that
\[
g_2(y)=\frac1{\sqrt{1+y}},\quad g_1(y)=-\frac8{15}\frac{(15d-4)(\sqrt[4]{1+y}-1)+y}{\sqrt{1+y}}.
\]
We do not need to find explicitly  $g_0,$ but only to know its existence for all $y+1>0.$ This follows from the differential equation associated to the coefficient of $x$ in $N.$ By taking these functions
\[
N(x,y)=-\frac13y(y+3)g_1(y)=\frac{8}{45}\frac{y+3}{\sqrt{1+y}}yM(y),\] where $M(y)=(15d-4)(\sqrt[4]{1+y}-1)+y.$ Hence the result will follow if we prove that $yM(y)\ge0$ for $y\ge-1,$ when $d\ge 1/5.$ This can be seen by studying the function $M$ and its derivative. Notice  that $M(-1)=3-15d\le0$ and that at $y\simeq0,$  $M(y)=\frac{15}4dy+O(y^2).$

\subsection{Nonexistence of limit cycles for some QS with an invariant cubic} It is known that there are two families of QS with an invariant cubic  and that have a (unique) limit cycle not contained on this curve, see for instance \cite{CG2,S2,ZK2}. We consider one of these families and give conditions on its parameters to have non existence of periodic orbits.

We take the system 
 \begin{equation}\label{eq:cubic}
	\begin{cases}
		\dot x=\dfrac12(1-x)x-\dfrac L 2 x^2+\dfrac12 u(-1+2x+xy),\\
		\dot y=-1-y-uy(1+y)+L(1+xy),
	\end{cases}
\end{equation}
where $L$ and $u$ are real parameters. It has the invariant algebraic cubic $f(x,y)=-1+2x+x^2y=0,$ with cofactor $k(x,y)=u-x,$ and for some values of $L$ and $u$ it also has a limit cycle. We prove:

\begin{prop}
	System \eqref{eq:cubic}, with $(1-u+Lu)(6-5u+9Lu)\ge0,$ does not have  periodic orbits.	
\end{prop}

\begin{proof}
	The proof follows exactly the same steps that the proof of Proposition \ref{th:para-ne}, but the computations are much easier and explicit. By taking in Theorem \ref{th:met}
	$$g(x,y)= -(u\alpha-u+1)(u\alpha-u-3)+8\alpha(Lu-u+1)x+4\alpha u^2 y,$$
with $\alpha$ real and such that
\[
-{u}^{2}{\alpha}^{2}+2u ( 6Lu-5u+5) \alpha+4L{u}^{
	2}-5{u}^{2}+6u-1=0,
\]
 it holds that $N(x,y)=r w^2(x,y),$ for some linear polynomial $w.$ The discriminant with respect to $\alpha$ of the above quadratic equation is $16u^2(1-u+Lu)(6-5u+9Lu)\ge0,$ and hence the proposition follows. 
\end{proof}

\section{Other families of planar systems}\label{se:of}

This final section is devoted to present some results on non existence of limit cycles for  other families of polynomial vector fields, most of them having  an invariant algebraic curve. Again we apply Theorem~\ref{th:met} and it includes cubic and Liénard systems.

\subsection{Some cubic systems}\label{ss:cs}

We start with the cubic system 
 \begin{equation}\label{eq:cub1}
	\begin{cases}
		\dot x=ax+y-(ax^2+xy+ay^2)(x+y)=P(x,y),\\
		\dot y=ay-(ax^2+xy+ay^2)(y-x)=Q(x,y),
	\end{cases}
\end{equation}
with the invariant algebraic curve $f(x,y)=x^2+y^2-1$ and cofactor $k(x,y)=-2(ax^2+xy+ay^2).$ This system is introduced in \cite{GLS} and there it is proved that it has exactly two limit cycles, counted with their own multiplicities, when $a\in I=(-1/2,(1-\sqrt{2})/2).$ Moreover, one of these limit cycles is always $f(x,y)=0,$ and it is the only one (and it is double) when $a=-1/4.$ Furthermore, it is not difficult to prove that for $a\in J:=(-\infty, (1-\sqrt{2})/2)\cup ((1+\sqrt{2})/2,\infty)$ the invariant circle $f(x,y)=0$ is always a limit cycle.

By using our approach we will show that it does not have  periodic orbits different of $f(x,y)=0,$ for most of the values $a\not\in I.$

\begin{prop} System \eqref{eq:cub1} does not have  periodic orbits different of $x^2+y^2-1=0,$ for $a\le -1/2$ and for $a\ge0.$
	\end{prop}	
\begin{proof} Notice first that $x^2+y^2-1=0$ is a periodic orbit if and only if the system does not have  critical points on it, and this happens if and only if $4a^2-4a-1>0,$ which is equivalent to $a\in J.$

First we apply Theorem \ref{th:met} with $g=1$ and $\alpha=1.$ The $N_{1,1}(x,y)=k(x,y)=-2(ax^2+xy+ay^2).$ Hence when $1\le4a^2,$ that is when $|a|\ge1/2,$ the result follows.

To prove the non existence of periodic orbits different of $f(x,y)=0,$ when $a\ge0$ we take in Theorem \ref{th:met}, $\alpha=-1-2a$ and $$g(x,y)=4a\big(2a^2x^2+2axy+(2a^2+1)y^2\big).$$
After several computations we get that
\[
N_{\alpha,g}(x,y)=16a^2(a^2x^2+2axy+(a^2+1)y^2)+8a(4a^2+1)(ax^2+xy+ay^2)^2.
\]
Since the discriminant of $a^2x^2+2axy+(a^2+1)y^2$ with respect to $x$ is $-4a^4y^2\le0$ and $a\ge0,$ it holds that $N_{\alpha,g}(x,y)\ge0$ and the result follows.  
\end{proof}	

The upper bound for the number of  limit cycles of system \eqref{eq:cub1} given in \cite{GLS} is obtained by transforming  it into an Abel differential equation. Here, we add some comments of another approach that also allows to study in detail the exact number of limit cycles of system~\eqref{eq:cub1}.  First we observe that if we introduce the function $$V(x,y)=(x^2+y^2-1)(ax^4+x^3y+2ax^2y^2+xy^3+ay^4-y^2)$$ it holds that
\[
\operatorname{div}\left(\frac{P(x,y)}{V(x,y)},\frac{Q(x,y)}{V(x,y)}  \right)=\frac{2(x^2+y^2-1)^2(ax^2+xy+ay^2)^2}{V^2(x,y)}\ge0.
\]
Hence by using the Bendixson-Dulac Theorem, we know that the number of limit cycles of \eqref{eq:cub1} is controlled by the number of holes of the set $\R^2\setminus\{V(x,y)=0\},$ see \cite[Sec 3.1]{GG2} for more details on the method and also for an application to a similar system. In particular it can be proved that two is an upper bound of the number of limit cycles  for all values of $a$ and also that limit cycles different of $f(x,y)=0,$ do exist if and only if $a\in I.$

Our second example is extracted from \cite[Sec. 4]{LZ} and writes as
 \begin{equation}\label{eq:cub2}
	\begin{cases}
		\dot x=-2h-y +x^2-4xy-y^2-\frac23x^3-3x^2y-5xy^2+\frac73y^3,\\
		\dot y=-h+x+\frac12x^2+3xy+2y^2-\frac43x^3-3x^2y+3xy^2+\frac{25}6y^3,
	\end{cases}
\end{equation}
with $h\in\R.$ It has the cubic invariant algebraic curve $f(x,y)=6h-3(x^2+y^2)+2x^3-9xy^2+2y^3=0,$ with cofactor $k(x,y)=(1-x+2y)(2x+y).$ 

\begin{prop} System \eqref{eq:cub2} does not have  periodic orbits.
\end{prop}	

\begin{proof}
By applying Theorem \ref{th:met} with $\alpha=-2$ and $g(x,y)=(1-x+2y)^2,$ we get that
$N_{\alpha,g}(x,y)=2k^2(x,y)\ge0$ and the proof follows.
\end{proof}

\subsection{Liénard systems}\label{ss:ls}   In this section, instead of using invariant algebraic curves we will use exponential factors, that are a particular case of generalized invariant  curves. More concretely, 
	an exponential factor will be a function  $f(x,y)=\exp(h(x,y)),$ with $h$ a polynomial and such that it holds that $\dot f (x,y)= k(x,y)f(x,y),$ as in \eqref{eq:iac} with $k$ being also a polynomial, called its cofactor. 

Recall that Theorem \ref{th:met} also holds if the $f$  of its statement is an exponential factor, instead of an invariant algebraic curve, see Remark~\ref{re:2}. A simple example of exponential factor  for the Liénard system \begin{equation}\label{eq:lie}
	\begin{cases}
		\dot x=y-xH(x),\\
		\dot y=-x
	\end{cases}
\end{equation}
is $f(x,y)=\exp(y),$ because $\dot f(x,y)=-x f(x,y)$ and $k(x,y)=-x.$ We prove:
 
\begin{prop}\label{pr:lie} If there exists $\alpha\in\R$ such that $\alpha x+2H(x)\not\equiv0$ does not change sign then system \eqref{eq:lie} does not have  periodic orbits.
\end{prop}

\begin{proof} From the above discussion we know that we can apply Theorem~\ref{th:met} by taking $f$ as the exponential factor $f(x,y)=\exp(y).$ We also take $g(x,y)= 2-2\alpha y-\alpha^2 x^2.$ Then some computations give that $N_{\alpha,g}(x,y)=\alpha^2x^2\big(\alpha x+2H(x)\big),$ and the result follows.
	\end{proof}

Next we will apply the above result to get effective conditions for the non existence of periodic solutions for two concrete families, that have limit cycles for some values of the parameters. Concretely, we will consider Liénard systems of degrees 3 and 5.

\begin{cor}Liénard system
	 \begin{equation*}
		\begin{cases}
			\dot x=y-a_1x-a_2x^2-x^3,\\
			\dot y=-x
		\end{cases}
	\end{equation*}
does not have  periodic solutions when $a_1\ge0.$
	\end{cor}
\begin{proof}
By using Proposition \ref{pr:lie} with $\alpha=4\sqrt{a_1}-2a_2$ and $H(x)=a_1+a_2x+x^2$ we get that
\[
\alpha x+2H(x)=2a_1+(\alpha +2a_2)x+2x^2= 2\big(\sqrt{a_1}+x\big)^2\ge0,
\]
and the result follows.
\end{proof}

To study the degree 5 case we will use the next result that follows from~\cite{Re}. In fact, in that paper more elaborated conditions are given to characterize the number of real and complex roots of quartic polynomials. Next lemma only presents a couple of useful and  well known simple consequences. 

\begin{lemma}\label{le:4}
	For the quartic real equation   \[P(x)=ax^4+bx^3+cx^2+dx+e=0,\quad a\ne0,\] define $\Delta=\operatorname{disc}(P),$ that is
	\begin{align*}
		\Delta=\operatorname{disc}_x(P)& =256 a^3 e^3 - 192 a^2 b d e^2 - 128 a^2 c^2 e^2 + 144 a^2 c d^2 e - 27 a^2 d^4 \\ 
		&\quad+ 144 a b^2 c e^2 - 6 a b^2 d^2 e - 80 a b c^2 d e + 18 a b c d^3 + 16 a c^4 e \\
		&\quad- 4 a c^3 d^2 - 27 b^4 e^2 + 18 b^3 c d e - 4 b^3 d^3 - 4 b^2 c^3 e + b^2 c^2 d^2,
	\end{align*}
	and $R=R(P)=3b^2-8ac.$ It holds that:
	\begin{enumerate}[(i)]
		\item If $\Delta<0$ then the polynomial $P$ has two real roots and two complex conjugated real roots.
		
		\item If $\Delta>0$ and  	 $R<0,$ then the polynomial $P$ does not have  real roots.

	\end{enumerate}	 
\end{lemma} 

\begin{cor}Liénard system
	\begin{equation*}
		\begin{cases}
			\dot x=y-a_1x-a_2x^2-a_3x^3-a_4x^4-x^5,\\
			\dot y=-x
		\end{cases}
	\end{equation*}
	does not have  periodic solutions when $U\ne0,$  $V>0$ and $a_1W>0,$ where
	\begin{align*}
	U & =	\big(-{a_{{4}}}^{4}+8a_{{3}}{a_{{4}}}^{2}-16{a_{{3}}}^{2}+64a_{{1}})^2 a_1, \quad V=8a_3-3a_4^2,\\
	W & =27a_{{1}}{a_{{4}}}^{4}-{a_{{3}}}^{3}{a_{{4}}}^{2}-108a_{{1}}a_{{3}
}{a_{{4}}}^{2}+3{a_{{3}}}^{4}+72a_{{1}}{a_{{3}}}^{2}+432{a_{{1}}
}^{2}
.\end{align*}
\end{cor}
\begin{proof} For the sake of shortness  we will write $\mathbf{A}$ to refer to the set of parameters $a_i, i=1,2,3,4.$ By Proposition \ref{pr:lie} we know that the result will follow if we can choose parameters $\mathbf{A}$ and a value of $\alpha$ such that
\[
P_\alpha(x)=2a_1+(\alpha +2a_2)x+2a_3x^2+2a_4x^3+2x^4\ge0,
\]	
for all $x\in\R.$ We will study the sign of $P_\alpha$ by applying Lemma \ref{le:4}. With this aim, we compute $R=R(\alpha,{\mathbf A}) =-4V<0$ and 
\[\Delta(\alpha,{\mathbf A})=\operatorname{disc}_x\big(P_\alpha\big)= \sum_{j=0}^4 q_j({\mathbf A})\alpha^j=:Q(\alpha),
\]
where $q_4=-108$ and the $q_i,i=0,1,2,3,$ are polynomials that we do not detail.  We claim that the inequalities of the statement imply that  there exist  $\mathbf{A}$ and  $\alpha$ such that  $\Delta(\alpha,{\mathbf A})>0.$ From this claim the result follows by item (ii) of the Lemma applied to $P_\alpha,$ because  $\Delta(\alpha,{\mathbf A})>0,$ $R<0$ and  then $P_\alpha(x)>0.$

To prove the claim we apply again the lemma, but in this case to the polynomial $Q(\alpha)$ and we will use item (i). The claim follows because
\[
\Delta(Q)=-2^{32}UW^3<0
\]
and hence the polynomial $Q(\alpha)$ has exactly two real simple roots and, as a consequence, it  takes positive (and also negative) values for suitable values of $\alpha.$  
 \end{proof}	

%\vspace{-0.4cm}

\section*{Acknowledgements}
The first author is partially supported by
the Minis\-te\-rio de Cien\-cia e Inno\-va\-ci\'{o}n (PID2019-104658GB-I00 grant),
by the grant Severo Ochoa
and Mar\'ia de Maeztu Program for Centers and Units of Excellence in
R\&D (CEX2020-001084-M)
and also by the Ag\`{e}ncia de Gesti\'{o} d'Ajuts Universitaris i de Recerca (2017 SGR 1617 grant).

%\vspace{-0.2cm}

%\vspace{-0.05cm}

\end{document}